\documentclass{amsart}
\usepackage{amsfonts,amsmath,amssymb,amsthm}
\usepackage{hyperref,cleveref}
\usepackage{graphicx}

\newtheorem{theorem}{Theorem}[section]
\newtheorem{lemma}[theorem]{Lemma}
\newtheorem{proposition}[theorem]{Proposition}
\newtheorem{corollary}[theorem]{Corollary}
\newtheorem{conjecture}[theorem]{Conjecture}
\newtheorem*{theorem*}{Theorem}

\theoremstyle{definition}
\newtheorem{definition}[theorem]{Definition}
\newtheorem{example}[theorem]{Example}
\newtheorem{example*}[theorem]{Example}

\theoremstyle{remark}
\newtheorem{remark}[theorem]{Remark}

\newcommand{\reg}{\textup{reg}}
\newcommand{\het}{\textup{ht}}
\newcommand{\rev}{\textup{rev}}
\newcommand{\gin}{\textup{gin}}
\newcommand{\depth}{\textup{depth}}
\newcommand{\iin}{\textup{in}}
\newcommand{\rank}{\textup{rank}}
\newcommand{\codim}{\textup{codim}}

\numberwithin{equation}{section}

\begin{document}
\begin{sloppypar}
\title{Restrictions on Hilbert coefficients give depths of graded domains}
\author{Cheng Meng}
\address{Yau Mathematical Sciences Center, Tsinghua University, Beijing 100084, China.}
\email{cheng319000@tsinghua.edu.cn}
\date{\today}

\begin{abstract}
In this paper, we prove that if $P$ is a homogeneous prime ideal inside a standard graded polynomial ring $S$ with $\dim(S/P)=d$, and for $s \leq d$, adjoining $s$ general linear forms to the prime ideal changes the $(d-s)$-th Hilbert coefficient by 1, then $\depth(S/P)=s-1$. This criterion also tells us about possible restrictions on the generic initial ideal of a prime ideal inside a polynomial ring.
\end{abstract}
\maketitle
\section{Introduction}
Let $k$ be an infinite field. When $k$ is algebraically closed, the famous Eisenbud-Goto conjecture claims that an inequality holds between several numerical invariants of a homogeneous prime $P$ in a polynomial ring over $k$:
\begin{conjecture}[\cite{eisenbud1984linear}]
Let $k$ be an algebraically closed field, $P \subset (x_1,\ldots,x_n)^2$ be a homogeneous ideal in $S=k[x_1,\ldots,x_n]$, then
$$\reg(P) \leq \deg(S/P)-\het(P)+1$$
\end{conjecture}
Here $\deg(S/P)$ is the multiplicity of $S/P$. This conjecture is proved for many special cases including curves and smooth surfaces, but it is false in general. The first counterexample is given by McCullough and Peeva in \cite{mccullough2018counterexamples} using Rees-like algebras. Their results imply that the regularity cannot be bounded by any polynomial in terms of $\deg(S/P)$.

When $P$ is a homogeneous prime ideal of $S$, we can take the graded reverse lexicographic order $<=<_{\rev}$ and talk about the generic initial ideal $\gin_<(P)$ of $P$. The above conjecture involves several invariants of $P$ including multiplicity, regularity and height, and all of them can be described using the combinatorial data of the generic initial ideal in a relatively simpler way. For example, let $J=\gin_<(P)$ and $G(J)$ be the monomial minimal generating set of $J$. Then by Bayer and Stillman's theorem in \cite{bayer1987criterion} and Eliahou and Kervaire's theorem in \cite{eliahou1990minimal}, $\reg(S/P)=\max\{\deg(u):u \in G(J)\}-1$ and $\depth(S/P)=n-\max\{i:x_i|u \in G(J)\}$, so a description of such a generic initial ideal may lead to similar inequalities of these invariants. Therefore, it makes sense to study the possible generic initial ideals of primes in a polynomial ring. Although we have a description of all the possible generic initial ideals (for example, see \cite{herzog2011monomial}) of a general homogeneous ideal, there may be more strict restrictions on the generic initial ideal of a homogeneous prime ideal. This paper gives such a restriction and shows that certain monomial ideals cannot be realized as the initial ideal or generic initial ideal of a homogeneous prime ideal. For example, in \Cref{4.3} we prove that in a polynomial ring in $n \geq 4$ variables, $(x_1^2,x_1x_2,x_2^2,x_1x_3,x_1x_4)$ cannot be the initial ideal or generic initial ideal of any prime ideal under $<_{\rev}$.  

Moreover, we give a numerical condition on the Hilbert coefficients of $S/P$ which gives $\depth(S/P)$ immediately.
\begin{theorem*}[See \Cref{5.4}]
Let $P$ be a homogeneous prime ideal in $S$, $\dim(S/P)=d$. Take $1 \leq s \leq d$ and let $r=n-s$. Choose $s$ general linear forms $l_1,\ldots,l_s$. Set $P_s=\pi_{l_1}\ldots\pi_{l_s}P \subset S(r)$. If $e_{d-s}(S/P)=e_{d-s}(S(r)/P_s)-(-1)^{d-s}$, then $\depth(S/P)=n-r-1$.
\end{theorem*}
Here $\pi_{l_1},\ldots,\pi_{l_s}$ are projections from a polynomial ring to another polynomial ring with one less variable; the concrete definition is in \Cref{sec2}.

\textbf{Conventions in this paper.} In this paper we work with prime ideals in a polynomial ring. Unless otherwise stated, we make the following assumptions throughout the paper: $k$ is an infinite field, $S=k[x_1,\ldots,x_n]$ is the standard graded polynomial ring over $k$ in $n$ variables, and for $m \in \mathbb{N}$, $S(m)=k[x_1,\ldots,x_m]$ is the standard graded polynomial ring over $k$ in $m$ variables (so $S=S(n)$). The maximal homogeneous ideal of $S$ and $S(m)$ are denoted by $\mathfrak{m}$ and $\mathfrak{m}(m)$ respectively. We assume $P$ is a homogeneous prime ideal in $S$ and $I$ is an arbitrary homogeneous ideal in $S$. We assume $\dim(S/P)=d$, $\codim(S/P)=n-d=c$. When we take a sequence of general linear forms, the number of general linear forms is $s$ with $1 \leq s \leq d$ and $r=n-s$.

\section{Notations}\label{sec2}
Before we describe the restrictions on generic initial ideals of prime ideals, we introduce some notations on polynomial rings, monomials, and monomial orders. If $n_1<n_2$, there is a natural embedding $S(n_1) \hookrightarrow S(n_2)$. If a linear form $l=c_1x_1+\ldots+c_nx_n \in S_1$ satisfies $c_n \neq 0$, then the map $\eta:S(n-1) \hookrightarrow S \to S/lS$ is an isomorphism. We define the map $\pi_l: S \to S/lS \xrightarrow[]{\eta^{-1}} S(n-1)$. If $l=x_n$, we write $\pi_n=\pi_{x_n}:S \to S(n-1)$. For a sequence of linear forms $l_1,\ldots,l_s$, we define the following map iterately:
$$\Pi_{l_1\ldots l_s}=\pi_{\overline{l_s}}\pi_{\overline{l_{s-1}}}\ldots\pi_{l_1}:S \to S(r)$$
where $\overline{l_i}$ is the image of $l_i$ under $\Pi_{l_1\ldots l_{i-1}}$. This definition does not make sense for all sequences of linear forms; it is well-defined only if the $x_n$-coefficient of $l_n$ is nonzero, the $x_{n-1}$-coefficient of $\overline{l_2}=\pi_{l_1}(l_2)$ is nonzero, and so on. The above conditions are given by the nonvanishing of finitely many polynomials in coefficients of $l_1,l_2,\ldots,l_s$, therefore $\Pi_{l_1\ldots l_s}$ exists on a Zariski open subset in the space of $s$ linear forms, that is, it exists for $s$ \textit{general} linear forms. And it also makes sense when $l_1=x_n,l_2=x_{n-1},\ldots,l_s=x_{r+1}$. Let
$$\Pi_r=\Pi_{x_n,x_{n-1},\ldots,x_{r+1}}=\pi_{r+1}\pi_{r+2}\ldots\pi_n:S \to S(r)$$
In particular, we define $\Pi_n=\textup{id}_{S}$. If $I \subset S$ is a homogeneous $S$-ideal, the saturation of $I$ is denoted by $I^{sat}=I:\mathfrak{m}^\infty$.

We recall the definition of the initial ideal and generic initial ideal taken from \cite{elias1998six} and \cite{herzog2011monomial}. Let $<$ be a monomial order on $S$. For $0 \neq f \in S$, we can write $f$ as a $k$-linear combination of monomials, that is, $f=\sum_u a_uu,a_u \in k$. Define the initial of $f$ with respect to $<$, denoted by $\iin_<(f)$, as the largest monomial $u$ such that $a_u \neq 0$. The initial ideal of $I$ is $\iin_<(I)=(\iin_<(f)|f \in I)$.

Now suppose $k$ is an infinite field, then $GL_n(k)$ can be given a Zariski topology which makes it an irreducible topological space. Every linear map $\alpha \in GL_n(k)$ induces a linear automorphism of $S$; by abusing the notation we say $\alpha \in \textup{Aut}_k(S)$. We have the following proposition:
\begin{proposition}[\cite{ginchar02} and \cite{gincharp}]
We fix a monomial order $<$ and an ideal $I$. There exists a nonempty Zariski open set $U \subset GL_n(k)$ such that for all $\alpha \in U$, $\iin_<(\alpha(I))$ is independent of the choice of $\alpha$.    
\end{proposition}
This common initial ideal is called the generic initial ideal of $I$, denoted by $\gin_<(I)$, or $\gin(I)$ if the order is clear.

Throughout this paper, we will work with the graded reverse lexicographic order $<=<_{\rev}$ with $x_1>x_2>\ldots>x_n$.
\begin{definition}
We say a monomial ideal $J$ is of Borel type if for any $i<j$ and a monomial $u \in J$ dividing $x_j$, there is some integer $s$ such that $x_i^su/x_j \in J$.    
\end{definition}
This is equivalent to the condition $J:(x_1,\ldots,x_i)^\infty=J:x_i^\infty$ for any $i$.
\begin{proposition}
For any homogeneous ideal $I$, $\gin(I)$ is always of Borel type. Moreover, $\dim(S/I)=d$ if and only if $\gin(I)$ contains pure powers of $x_1,\ldots,x_c$ but not pure powers of $x_{c+1},\ldots,x_n$, where $c=n-d$. 
\end{proposition}
We also recall the definition of Hilbert coefficients taken from \cite{eisenbud2013commutative}. Let $M$ be a finitely generated graded $S$-module. The function $H_M: \mathbb{Z} \to \mathbb{N}, H_M(m)=\dim_k (M_m)$ is called the Hilbert function of $M$, and the power series $h_M(t)=\sum_{i \in \mathbb{Z}}H_M(i)t^i$ is called the Hilbert series of $M$. It is well-known that $h_M(t)=q_M(t)/(1-t)^{\dim M}$ with $q_M(t) \in \mathbb{Z}[t], q_M(1) \neq 0$.
\begin{definition}
Let $M$ be a finitely generated graded $S$-module, $q_M(t)=h_M(t)(1-t)^{\dim(M)} \in \mathbb{Z}[t]$. Expand $q_M(t)$ as a linear combination of powers of $t-1$:
$$q_M(t)=e_0+e_1(t-1)+e_2(t-1)^2+\ldots$$
The coefficient $e_i$ is called the $i$-th Hilbert coefficient of $M$.
\end{definition}
Since $q_M(t) \in \mathbb{Z}[t]$, all the Hilbert coefficients are integers.

\section{Generic initial ideal under projection and saturation}
For a monomial $u=x_1^{e_1}x_2^{e_2}\ldots x_n^{e_n}$, let $\phi_i(u)$ denote $u/x_i^{e_i}$, that is, we eliminate all $x_i$'s from the factors of $u$. Let $\Phi_r(u)=\phi_{r+1}\phi_{r+2}\ldots\phi_n(u)$. If $J=(u_1,\ldots,u_m)$ is a monomial ideal in $S$ with monomial minimal generators $u_1,\ldots,u_m$, we define $\phi_i(J)=(\phi_i(u_1),\ldots,\phi_i(u_m))=J:x_i^\infty$, $\Phi_r(J)=(\Phi_r(u_1),\ldots,\Phi_r(u_m))=J:(x_{r+1}\ldots x_n)^\infty$, $\overline{\Phi_r}(J)=\pi_{r+1}\phi_{r+1}\pi_{r+2}\phi_{r+2}\ldots\pi_n\phi_n(J)$. Note that $\Phi_r(J)$ is an $S$-ideal and $\overline{\Phi_r}(J)$ is an $S(r)$-ideal, and these two ideals are generated by the same set of monomials in $S(r)$.

\begin{remark}
If $i \neq j$, $\pi_i\phi_j=\phi_j\pi_i$. Therefore, if $J$ is generated by $u_1,\ldots,u_m$, then $\pi_i\phi_j(J)=\phi_j\pi_i(J)$ is generated by $\phi_j(u_k)$ where $u_k$ is not divisible by $x_i$. Thus in the expression of $\overline{\Phi_r}(J)$, we can interchange all $\phi$'s and $\pi$'s if we only move $\pi$ to the left; that is, $\overline{\Phi_r}(J)=\Pi_r\Phi_r(J)$.    
\end{remark}
The following two properties show that the generic initial ideal in reverse lexicographic order behaves well under the projection map and saturation:
\begin{proposition}[\cite{elias1998six}, Proposition 2.14 and \cite{bayer1987criterion}, Lemma 2.2]\label{3.2}
Let $I$ be a homogeneous ideal in $S$, $l$ be a general linear form, then $\gin(\pi_l(I))=\pi_n(\gin(I))$.    
\end{proposition}
\begin{remark}
Here $\pi_l(I)$ is an ideal in $S(n-1)$, so the generic initial ideal is well-defined. $\pi_n(\gin(I))$ is also an ideal of $S(n-1)$. Thus this equality makes sense because it compares two ideals in the same ring $S(n-1)$.    
\end{remark}
\begin{proposition}\label{3.4}
$\gin(I^{sat})=\gin(I):x_n^\infty=\gin(I)^{sat}$.    
\end{proposition}
\begin{proof}
The first equality is proved in \cite{elias1998six}, Proposition 2.21 and \cite{eisenbud2013commutative}, Proposition 15.24. The second statement is true because $\gin(I)$ is of Borel type.    
\end{proof}
For simplicity, we will denote the saturation operation $I \to I^{sat}$ by $\sigma$.

Let $I$ be a saturated homogeneous ideal in $S$. Let $l$ be a linear form such that the $x_n$-coefficient of $l$ is nonzero. We call the ideal $\sigma\pi_l(I)=(\pi_l(I))^{sat}$ the section with one hyperplane. It is a saturated ideal in $S(n-1)$. If we have $s$ general linear forms $l_1,\ldots,l_s$, then inductively we can define the ideal of section with $s$ hyperplanes: the ideal of section with one hyperplane is $I_1=\sigma\pi_{l_1}(I)$ which is an ideal in $S(n-1)$; the section with two hyperplanes $I_2=\pi_{\overline{l_2}}(I_1)^{sat}$ where $\overline{l_2}=\pi_{l_1}(l_2)$, and so on, until we get $I_s=\pi_{\overline{l_s}}(I_{s-1})^{sat}$ which is an ideal in $S(r)$. In algebraic geometry, we can consider the algebraic set $X \subset \mathbb{P}^{n-1}$ corresponding to $I$; the intersection of $X$ with $s$ general hyperplanes defined by $l_1,\ldots,l_s$ is $X_s$, then $I_s$ is the defining ideal of $X_s$ inside the linear subvariety $\mathbb{P}^{n-s-1}$ defined by $l_1,\ldots,l_s$. The coordinate ring of the linear subvariety is $S/(l_1,\ldots,l_s)=S(r)$.
\begin{proposition}\label{3.5}
$\gin(I_s)=\sigma\pi_{r+1}\sigma\pi_{r+2}\ldots\sigma\pi_{n}(\gin(I))$.    
\end{proposition}
\begin{proof}
Apply \Cref{3.2} and \Cref{3.4} inductively.    
\end{proof}
\begin{lemma}
Let $J$ be a monomial ideal of Borel type and $1 \leq r \leq n$. Then $\Pi_r(J)$, $\Phi_r(J)$, $\overline{\Phi_r}(J)$ are all of Borel type.    
\end{lemma}
\begin{proof}
Suppose $J=(u_1,\ldots,u_m)$, then $\Phi_r(J)=(\Phi_r(u_1),\ldots,\Phi_r(u_m))$. Every minimal generator of $\Phi_r(J)$ is of the form $\Phi_r(u_k)$ for some $k$. Choose $i$ such that $x_i|\Phi_r(u_k)$ and $j<i$. Then $i \leq r$ because $\Phi_r(u_k)$ is only divisible by a subset of $\{x_1,\ldots,x_r\}$. In this case $x_i|u_k$. Since $J$ is of Borel type, there exists $t$ such that $x_j^tu_k/x_i \in J$, so there is another minimal generator $u_{k'}|x_j^tu_k/x_i$. Since $j<i\leq r$, $\Phi_r(u_{k'})|\Phi_r(x_j^tu_k/x_i)=x_j\Phi_r(u_k)/x_i$. This means $\Phi_r(J)$ is still of Borel type. For $\Pi_r$, we have $\Pi_r(J)=(\Pi_r(u_1),\ldots,\Pi_r(u_m))$, and the set of minimal generators of $\Pi_r(J)$ is just the set of $\Pi_r(u_k)=u_k$'s where $\Pi_r(u_k) \neq 0$. Choose $i$ such that $x_i|\Pi_r(u_k)$ and $j<i$, then $i<r$ because $\Pi_r(u_k) \neq 0$ is only divisible by a subset of $\{x_1,\ldots,x_r\}$. Then we can use the same argument as $\Phi_r(J)$ to prove $\Pi_r(J)$ is of Borel type. The last statement is true by the previous two statements because $\overline{\Phi_r}=\Pi_r\Phi_r$.   
\end{proof}
\begin{proposition}\label{3.7}
Let $J$ be a saturated monomial ideal of Borel type of $S$. Then
\begin{enumerate}
\item $\pi_{r+1}\sigma\pi_{r+2}\ldots\sigma\pi_n(J)=\pi_{r+1}\phi_{r+1}\ldots\pi_n\phi_n(J)=\overline{\Phi_r}(J)$.
\item $\sigma\pi_{r+1}\sigma\pi_{r+2}\ldots\sigma\pi_n(J)=\phi_r\pi_{r+1}\phi_{r+1}\ldots\pi_n\phi_n(J)=\phi_r\overline{\Phi_r}(J)$.
\end{enumerate}
\end{proposition}
\begin{proof}
$J$ is of Borel type and saturated, so $J=J^{sat}=J:x_n^\infty=\phi_n(J)$ and $\pi_n(J)=\pi_n\phi_n(J)$, which means (1) is true for $r=n-1$ and (2) is true for $r=n$. It is obvious that (2) for $r$ implies (1) for $r-1$, so by induction it suffices to show (1) for $r$ implies (2) for $r$. By lemma before we see $\overline{\Phi_r}(J)$ is of Borel type, so if (1) is true for $r$, then $\sigma\pi_{r+1}\sigma\pi_{r+2}\ldots\sigma\pi_n(J)=(\pi_{r+1}\sigma\pi_{r+2}\ldots\sigma\pi_n(J))^{sat}=(\overline{\Phi_r}(J))^{sat}=\phi_r(\overline{\Phi_r}(J))$, so (2) is true for $r$ and we are done.    
\end{proof}
The above proposition means that for monomial ideals of Borel type, we can replace the saturation operation $\sigma$ by taking saturation with respect to a single variable $\phi_i$.

\section{Initial and generic initial ideal of prime ideals}
This section describes restrictions on the generic initial ideal of a homogeneous prime ideal inside $S$.
\begin{theorem}\label{4.1}
Let $P$ be a homogeneous prime ideal in $S$ and $J=\iin(P)$. Assume $J$ is of Borel type, and $\overline{\Phi_r}(J)=\Pi_r(J)+u$ for some $1 \leq r \leq n-1$ and $u \in \Pi_r(J):\mathfrak{m}(r)$. Then:
\begin{enumerate}
\item Either $u=1$, $(x_1,\ldots,x_r) \subset J$, or $J$ is generated by $J\cap k[x_1,\ldots,x_r]$ and one extra generator $v$.
\item If $u \neq 1$, then the extra generator $v=ux_{r+1}^e$ for some $e>0$.
\item If $u \neq 1$, then $J$ is generated by monomials inside $k[x_1,\ldots,x_{r+1}]$.
\end{enumerate}
\end{theorem}
\begin{proof}
Let $f \in P$ such that $\iin(f)$ is a minimal generator of $J$. We claim $f$ is an irreducible polynomial. If $f$ is not irreducible, then $f=f_1f_2$, $f_1,f_2$ are not constants. Then $\iin(f)=\iin(f_1)\iin(f_2)$ with $\iin(f) \neq \iin(f_1)$ or $\iin(f_2)$. Since $P$ is prime and $f \in P$, $f_1 \in P$ or $f_2 \in P$, which implies $\iin(f_1) \in J$ or $\iin(f_2) \in J$, which contradicts the minimality of $\iin(f)$, so $f$ is irreducible. Let $u_1,\ldots,u_m,v_1,\ldots,v_t$ be the monomial minimal generator of $J$ with $u_i \in k[x_1,\ldots,x_r]$, $v_j \notin k[x_1,\ldots,x_r]$. We write $v_j=w_jz_j$ where $w_j \in k[x_{r+1},\ldots,x_n]$, $z_j \in k[x_1,\ldots,x_r]$. Then $\Phi_r(u_i)=\Pi_r(u_i)=u_i$ and $\Phi_r(v_j)=z_j$, $\Pi_r(v_j)=0$. We have $\Pi_r(J)=(u_i)$, $\Phi_r(J)=(u_i,u)=(u_i,z_j)$. We claim $z_j=u$ for all $j$. Otherwise $z_j \neq u$, then $z_j \in \mathfrak{m}u \subset (u_i)$, so there exists $i$ such that $u_i|z_j$, so $u_i|v_j$, which contradicts the fact that $u_i,v_j$ are two minimal generators. So $z_j=u$ for all $j$, and $v_j=uw_j$. We have seen the number of $v_j$'s, $t \geq 1$, since if there are no $v_j$'s then $\overline{\Phi_r}(J)=\Pi_r(J)$ which contradicts our assumption. If $t=1$, then the second part of (1) is satisfied. Now we assume $t \geq 2$, so there are at least two $v_j$'s, say $v_1$ and $v_2$. Take $f_1,f_2 \in P$ such that $\iin(f_1)=v_1$, $\iin(f_2)=v_2$. We may assume $f_1,f_2$ come from a reduced Gr\"obner basis, that is, any monomial appearing in $f_1$ or $f_2$ is not in $J$ except for their initial terms. By the previous arguments, $f_1,f_2$ are irreducible. We write $f_1=up_1+q_1$ and $f_2=up_2+q_2$ such that $up_1$ is the sum of all terms appearing in $f_1$ that are divisible by $u$, and $q_1$ is the sum of all other terms; similarly for $q_2,p_2$. Then $\iin(p_1)=w_1$, $\iin(p_2)=w_2$. Besides, $u \in (u_1,\ldots,u_s):\mathfrak{m}(r)$, so any other monomial $m$ appearing in $p_1$ is not divisible by $x_1,\ldots,x_r$, otherwise $mu$ divides some $u_i \in J$, so $mu$ cannot appear as a non-initial term of $uq_1$ or $f_1$. This means $p_1 \in k[x_{r+1},\ldots,x_n]$. In the same way we can prove $p_2 \in k[x_{r+1},\ldots,x_n]$. Consider the polynomial $F=p_2f_1-p_1f_2=p_2q_1-p_1q_2$. Take any term $m_1$ of $p_2$ and $m_2$ of $q_1$. Then $m_1 \in k[x_{r+1},\ldots,x_n]$. We also have $m_2\notin (u_1,\ldots,u_m)$ since $f_2$ is part of a reduced Gr\"obner basis, and $m_2 \notin (u)$ by definition of $q_1$. So $m_2 \notin (u_1,\ldots,u_m,u)=\Phi_r(J)=J:(x_{r+1}\ldots x_n)^\infty$. So $m_1m_2 \notin J$. This is true for any choice of $m_1m_2$. Similarly the product of any term of $p_1$ and any term of $q_2$ does not lie in $J$. So any possible term of $F$ does not lie in $J=\iin(P)$. But $F \in P$, so $F=0$ and $p_2f_1=p_1f_2$. Since $f_1$ is irreducible, we have $f_1|p_1$ or $f_1|f_2$. If $f_1|p_1$ then $\iin(f)=v_1=uw_1|w_1=\iin(p_1)$, which means $f_1$ and $p_1$ differ by a constant and $u=1 \in \Pi_r(J):\mathfrak{m}(r)$, so $(x_1,\ldots,x_r) \subset J$ and the first part of (1) holds. Similarly if $f_2|p_2$ the first part of (1) also holds. If both of the above are false, then $f_1|f_2$ and $f_2|f_1$, so $f_1$ and $f_2$ differ by a constant, so they have the same initial, which is a contradiction because we assume $v_1 \neq v_2$. So (1) is proved. Suppose (2) is false, and $v=uw \in J$ where $w$ is not a pure power of $x_{r+1}$. Then $x_j|w$ for some $j \geq r+1$. Since $J$ is of Borel type, there exists $e>0$ such that $v'=uwx^e_{r+1}/x_j \in J$. Pick a monomial minimal generator $v''$ dividing $v'$. $v''$ does not divide $v$ because $v'$ does not divide $v$, and $v$ does not divide $v''$ by minimality of $v$. Then $v=uw$, $v''=uw''$, $w,w'' \in k[x_{r+1},\ldots,x_n]$, and $w,w''$ do not divide each other which is contradictory to the second part of (1), so (2) is true. (3) is a corollary of (1) and (2). 
\end{proof}
\begin{theorem}\label{4.2}
Let $P$ be a homogeneous prime ideal in $S$ and $J=\gin(P)$. Assume $\overline{\Phi_r}(J)=\Pi_r(J)+u$ for some $u \in \Pi_r(J):\mathfrak{m}(r)$ and some $1 \leq r \leq n-1$. Then:
\begin{enumerate}
\item Either $u=1$, $(x_1,\ldots,x_r) \subset J$, or $J$ is generated by $J\cap k[x_1,\ldots,x_r]$ and one extra generator $v$.
\item If $u \neq 1$, then the extra generator $v=ux_{r+1}^e$ for some $e>0$.
\item If $u \neq 1$, then $J$ is generated by monomials inside $k[x_1,\ldots,x_{r+1}]$.
\end{enumerate}
\end{theorem}
\begin{proof}
The generic initial ideal of any ideal is always of Borel type. By definition of the generic initial ideal, $J=\iin(\alpha(P))$ for some $\alpha \in GL_n(k)$ and $\alpha(P)$ is still a prime ideal in $S$, so the conclusion follows from \Cref{4.1}.    
\end{proof}
\begin{example}\label{4.3}
The above theorem and corollary rule out some monomial ideals as the initial ideal of a prime ideal. For example, let $J=(x_1^2,x_1x_2,x_2^2,x_1x_3,x_1x_4) \subset k[x_1,\ldots,x_n], n \geq 4$. Then $J$ is of Borel type. Let $r=2$, we see $\Phi_r(J)=(x_1,x_2)^2$, $\Pi_r(J)=(x_1,x_2^2)$, and $x_1 \in (x_1,x_2)^2:x_1$. However, $x_1 \neq 1$, and $J$ has a minimal generator that can be divided by $x_4$. So $J$ violates conclusion (1) of \Cref{4.1}, and it cannot be the initial ideal or generic initial ideal of any prime ideal.    
\end{example}

\section{Difference in the Hilbert coefficients}
The previous section talks about restrictions on $\gin(P)$ for a prime ideal $P$. However, the generic initial ideal of an ideal is hard to capture in practice as its computation requires the information of some unknown $\alpha \in GL_n(k)$. It would be easier to describe the restrictions using Hilbert coefficients which is totally computable from the Hilbert function of the quotient ring. We want to see how the Hilbert coefficients change after going modulo $s$ general linear forms.

\begin{lemma}\label{5.1}
Let $M,N$ be two graded modules over $S$ with $\dim M=\dim N=m$. Assume there is an exact sequence
$$0 \to M_1 \to N \to M \to M_2 \to 0$$
Let $\dim(M_1)=d_1$, $\dim(M_2)=d_2$, $D=\max\{d_1,d_2\}$. Then:
\begin{enumerate}
\item $D \leq m$.
\item If $i<m-D$, then $e_i(N)=e_i(M)$.
\item If $i=m-D$, then $e_i(N)=e_i(M)+(-1)^i\delta_{D,d_1}e_0(M_1)-(-1)^i\delta_{D,d_2}e_0(M_2)$.
\end{enumerate}
\end{lemma}
\begin{proof}
We have $\dim(M) \geq \dim(M_2)$ and $\dim(N) \geq \dim(M_1)$, so $D \leq m$. The Hilbert series is additive on short exact sequences, so $h_N(t)=h_M(t)+h_{M_1}(t)-h_{M_2}(t)$. Multiply both sides by $(1-t)^m$, we get $q_N(t)=q_M(t)+q_{M_1}(t)(1-t)^{m-d_1}-q_{M_2}(t)(1-t)^{m-d_2}$. Now we expand both sides in terms of powers of $t-1$ and look at the coefficients of $1,t-1,\ldots,(t-1)^{m-D}$.    
\end{proof}
Now let $J$ be a monomial idea of Borel type in $S$, $\dim(S/J)=d$, and $r$ is an integer satisfying $r \geq n-d$. Let $J_1=(J \cap k[x_1,\ldots,x_r])S$ and $J_2=\Phi_r(J)$. Then $J_1 \subset J \subset J_2$ with $\Pi_r(J_1)=\Pi_r(J)$.
\begin{lemma}\label{5.2}
We have:
\begin{enumerate}
\item $e_i(S/J)=e_i(S/J_2)$ for $0 \leq i \leq r+d-n$;
\item $e_i(S/J)=e_i(S/J_1)$ for $0 \leq i \leq r+d-n-1$, and
\begin{align*}
e_{r+d-n}(S/J_1)-e_{r+d-n}(S/J)=(-1)^{r+d-n}\rank_{k[x_{r+1},\ldots,x_n]}(J_2/J_1)\\
=(-1)^{r+d-n}\dim_k(\Pi_r(J_2)/\Pi_r(J_1)).
\end{align*}
\end{enumerate}
\end{lemma}
\begin{proof}
For any $1 \leq i \leq n-1$, $\Phi_{i+1}(J) \subset \Phi_i(J)=\Phi_i(J):x_{i+1}^\infty$. For any monomial $u \in \Phi_i(J)\backslash\Phi_{i+1}(J)$, there is $e>0$ such that $ux^e_{i+1} \in \Phi_{i+1}(J)$. Since $J$ is of Borel type, $\Phi_{i+1}(J)$ is also of Borel type, so for any $j \leq i$, there exists some $e'>0$ such that $ux_j^{e'}\in\Phi_{i+1}(J)$. This means $\Phi_i(J)/\Phi_{i+1}(J)$ is annihilated by some power of $(x_1,\ldots,x_{i+1})$, so $\dim(\Phi_i(J)/\Phi_{i+1}(J)) \leq n-i-1$. Now consider the sequence
$$0 \to \Phi_i(J)/\Phi_{i+1}(J) \to S/\Phi_{i+1}(J) \to S/\Phi_i(J) \to 0.$$
By the argument above and \Cref{5.1}, we know
$$e_j(S/\Phi_i(J))=e_j(S/\Phi_{i+1}(J)), \forall j<d-(n-i-1)=d-n+i+1$$
So $e_j(S/J)=e_j(S/\Phi_r(J))$ for $j \leq r+d-n$. So (1) is proved. We know $S/J_1$ and $S/J_2$ are both free $k[x_{r+1},\ldots,x_n]$-module since $J_1,J_2$ are both generated by monomials in $x_1,\ldots,x_r$. For any $u \in J_2$, we write $u=vw$ where $v \in k[x_1,\ldots,x_r], w \in k[x_{r+1},\ldots,x_n]$. By definition of Borel type, for any $j \leq r$, there exists $u'=vx^e_j \in J$, which implies $u' \in J_1$. Therefore, $J_2 \subset J_1:\mathfrak{m}(r)^\infty$, so $J_2/J_1$ is a free $k[x_{r+1},\ldots,x_n]$-module of finite rank. Its Krull dimension is $\dim(J_2/J_1)=n-r$ and $e_0(J_2/J_1)=\rank_{k[x_{r+1},\ldots,x_n]}(J_2/J_1)=\dim_k(\Pi_r(J_2)/\Pi_r(J_1))$. Apply \Cref{5.1} to the exact sequence $0 \to J_2/J_1 \to S/J_1 \to S/J_2 \to 0$, we know $e_i(S/J_2)=e_i(S/J_1)$ for $0 \leq i \leq r+d-n-1$ and $e_{r+d-n}(S/J_1)=e_{r+d-n}(S/J_2)+(-1)^{r+d-n}e_0(J_2/J_1)$.
\end{proof}
\begin{lemma}\label{5.3}
Let $J$ be a monomial ideal of Borel type in $S$ and $\dim(S/J) \geq n-r$. Then $\overline{\Phi_r}(J)=\Pi_r(J)+u$ for $u \in \Pi_r(J):\mathfrak{m}(r)$ if and only if
$$\rank_{k[x_{r+1},\ldots,x_n]}(J_2/J_1)=\dim_k(\Pi_r(J_2)/\Pi_r(J_1))=1.$$
\end{lemma}
\begin{proof}
Since $J_2/J_1$ is free over $k[x_{r+1},\ldots,x_n]$, its rank is one if and only if modulo $(x_{r+1},\ldots,x_n)$ the quotient is a one-dimensional $k$-vector space, if and only if $\dim_k(\Pi_r(J_2)/\Pi_r(J_1))=1$.
\end{proof}
\begin{theorem}\label{5.4}
Let $P$ be a homogeneous prime ideal in $S$, $\dim(S/P)=d$. Take $1 \leq s \leq d$ and let $r=n-s$. Choose $s$ general linear forms $l_1,\ldots,l_s$. Denote $P_s=\pi_{l_1}\ldots\pi_{l_s}P \subset S(r)$. Then:
\begin{enumerate}
\item $e_i(S/P)=e_i(S(r)/P_s)$ for $0 \leq i \leq d-s-1$.
\item If $e_{d-s}(S/P)=e_{d-s}(S(r)/P_s)-(-1)^{d-s}$, then $\depth(S/P)=n-r-1$.
\end{enumerate}
\end{theorem}
\begin{proof}
Let $J=\gin(P), J_1=(J\cap k[x_1,\ldots,x_r])S, \bar{J_1}=\Pi_r(J)$. Then $\bar{J_1}=\gin(P_s)$ by \Cref{3.5}. Since $J_1$ is a monomial ideal in $S(r)$, $S/(J_1,x_{r+1},\ldots,x_n)=S(r)/\bar{J_1}$ and $(x_{r+1},\ldots,x_n)$ is a regular sequence of linear forms on $S/J_1$. Therefore, $q_{S/J_1}=q_{S(r)/\bar{J_1}}$ and $e_i(S/J_1)=e_i(S(r)/\bar{J_1})$ for any $i$. Now taking generic initial ideal does not change the Hilbert series, so the Hilbert coefficients are the same, so we have $e_i(S/P)=e_i(S/P_s)$ for $0 \leq i \leq d-s-1$. If $e_{d-s}(S/P)=e_{d-s}(S(r)/P_s)+(-1)^{d-s}$, then $e_{d-s}(S/J)=e_{d-s}(S/J_1)-(-1)^{d-s}=e_{d-s}(S(r)/\bar{J_1})-(-1)^{d-s}$. By \Cref{5.3} we know $\overline{\Phi_r}(J)=\Pi_r(J)+u$ with $u \in \Pi_r(J):\mathfrak{m}(r)$. Now by \Cref{4.1} we know $\depth(S/P)=n-r-1$.
\end{proof}

\section{Height 1 prime ideal and Artinian hyperplane section}
When the height of the prime is one, the possible initial ideal is very simple.
\begin{proposition}\label{6.1}
Let $n \geq 3$ and $J$ be a monomial ideal in $S$. Then there exists a prime ideal $P$ of $S$ such that $J=\gin(P)$ for some height $1$ prime $P$ if and only if $J=x_1^e$ for some $e>0$.    
\end{proposition}
\begin{proof}
Since $S$ is a UFD, a height $1$ prime is just a principal ideal generated by an irreducible ideal $f$. Now for a general linear change of coordinate $\alpha$, $\iin(\alpha f)=x_1^{\deg(f)}$. Conversely for any degree $d$, we have a polynomial $x_1^d-x_2^{d-1}x_3$. Apply the Eisenstein criterion for the ideal $(x_3)$, we know it is irreducible, and its generic initial monomial is just $x_1^d$.    
\end{proof}
\begin{corollary}\label{6.2}
Let $J=\gin(P)$. Suppose $\het(J)=1$, then $J=x_1^e$ for some $e>0$.    
\end{corollary}
\begin{proof}
$\dim(S/J)=\dim(S/P)=n-1$. So $\het(P)=1$ because $S$ is catenary and $P$ is prime, and \Cref{6.2} follows from \Cref{6.1}.    
\end{proof}
Now we consider an artinian hyperplane section of $S/P$, that is, we take section with $s=d$ hyperplanes. In this case $c=n-d=n-s=r$. We have:
\begin{proposition}
Let $\dim(S/P)=d$, $J=\gin(P)$. Then:
\begin{enumerate}
\item $x_1,\ldots,x_c \in \sqrt{J}$ and $x_{c+1},\ldots,x_n \notin \sqrt{J}$.
\item $R=k[x_{c+1},\ldots,x_n]$ is a Noether normalization of $S/P$ and $S/J$.
\item $S/\Phi_c(J)$ is free over $R$ and $\Phi_c(J)/J$ is $R$-torsion.
\item $\deg(S/P)=\deg(S/J)=\rank_R(S/J)=\rank_R(S/\Phi_c(J))=\dim_k(S/\Phi_c(J)+(x_{c+1},\ldots,x_n))$.
\end{enumerate}
\end{proposition}
In this case, by \Cref{5.2} and \Cref{5.3}, $\overline{\Phi_c(J)}=\Pi_c(J)+u$ if and only if $\deg(S/J)+1=\deg(S/\Phi_c(J))+1=\deg(S(c)/\Pi_c(J))$. That is, the degree increases by exactly one after applying $\Pi_c$. So by \Cref{4.1} we have:
\begin{theorem}
Let $P$ be a prime ideal in $S$, $\dim(S/P)=d$, $c=n-d$, $J=\gin(P)$, and $J$ satisfies $\deg(S/\Phi_c(J))+1=\deg(S(c)/\Pi_c(J))$. Then $J=(J\cap k[x_1,\ldots,x_{c+1}])S$, and $x_{c+1}$ appears in the minimal generator of $J$. Moreover, $\depth(S/P)=d-1$, and $S/P$ is almost Cohen-Macaulay.   
\end{theorem}
\begin{remark}
This is a generalization of Theorem 5.1 in Kwak's paper \cite{cuong2020reduction}, where we have $S(c)/\Pi_c(J)=S(c)/\mathfrak{m}(c)^{r+1}$. Here $r$ is the reduction number of $S/P$.    
\end{remark}

\section*{Acknowledgements}
The author would like to thank Giulio Caviglia for introducing this problem and providing references. The author would also like to thank Matteo Varbaro for helpful suggestions.

\bibliographystyle{plain}
\bibliography{reference}
\end{sloppypar}
\end{document}